\documentclass[12pt]{article}
\usepackage{amsthm,amsmath,amssymb,color}
\usepackage[noadjust,sort]{cite}

\normalsize

\newtheorem{thm}{Theorem}
\newtheorem{cor}[thm]{Corollary}
\newtheorem{lemma}[thm]{Lemma}

\renewcommand{\dfrac}[2]{\lower0.15ex\hbox{\large$\frac{#1}{#2}$}}
\renewcommand{\(}{\bigl(}
\renewcommand{\)}{\bigr)}

\newcommand{\dig}{\mathaccent"017E} % small over-arrow; \vec was redefined by the style
\newcommand{\B}{{\mathcal B}} 
\newcommand{\digB}{\dig{\mathcal B}}
\newcommand{\D}{{\mathcal D}}
\newcommand{\G}{{\mathcal G}} 
\newcommand{\digG}{\dig{\mathcal G}}
\newcommand{\I}{{\mathcal I}} \newcommand{\digI}{\dig{\mathcal I}}
\newcommand{\V}{{\mathcal V}} \newcommand{\digV}{\dig{\mathcal V}}

\newcommand{\eps}{\varepsilon}

\newcommand{\boldital}[1]{{\textit{\textbf{#1}}}}
\newcommand{\svec}{\boldital{s}}
\newcommand{\tvec}{\boldital{t}}
\newcommand{\xvec}{\boldital{x}}
\newcommand{\Svec}{\boldital{S}}
\newcommand{\Tvec}{\boldital{T}}
\newcommand{\Xvec}{\boldital{X}}

\newcommand{\oo}{\tilde{o}}

\newcommand{\sprod}{\prod_{i=1}^m \binom{n}{s_i}}
\newcommand{\tprod}{\prod_{j=1}^n \binom{m}{t_j}}
\newcommand{\dsprod}{\prod_{i=1}^n \binom{n{-}1}{s_i}}
\newcommand{\dtprod}{\prod_{j=1}^n \binom{n{-}1}{t_j}}

\newcommand{\Var}{\operatorname{Var}}
\newcommand{\Prob}{\operatorname{Prob}}
\newcommand{\Reals}{{\mathbb{R}}}

\newcommand{\expect}{\operatorname{\mathbb E}}
\newcommand{\abs}[1]{\lvert#1\rvert} \let\card=\abs
\newcommand{\Abs}[1]{\bigl|#1\bigr|} 
\newcommand{\ssqrt}[1]{\sqrt{\smash{#1}\vrule width0pt height2.0ex depth0.3ex}}
\newcommand{\stevent}{\Svec=\svec\wedge\Tvec=\tvec}

\DeclareMathOperator{\Bin}{Bin}
\newcommand{\suchthat}{\mathrel{:}}
\newcommand{\nicebreak}{\vskip 0pt plus 50pt\penalty-300\vskip 0pt plus -50pt }

\begin{document}

\date{}
\title {Degree sequences of random digraphs\\ and bipartite graphs}

\author{
Brendan~D.~McKay\vrule width0pt height2ex\thanks
 {Research supported by the Australian Research Council.}
 ~and~ 
 Fiona Skerman\vrule width0pt height2ex\thanks
 {Current address: University of Oxford, Department of Statistics, Oxford OX1 3TG, UK}\\
\small Research School of Computer Science\\[-0.8ex]
\small Australian National University\\[-0.8ex]
\small Canberra ACT 0200, Australia\\[-0.3ex]
\small\texttt{bdm@cs.anu.edu.au, ~skerman@stats.ox.ac.uk}
}

\maketitle

\begin{abstract}
We investigate the joint distribution of the vertex degrees in three
models of random bipartite graphs.
Namely, we can choose each edge with a specified
probability, choose a specified number of edges, or specify the vertex degrees
in one of the two colour classes.

This problem can alternatively be described in terms of the
row and sum columns of random binary matrix or the in-degrees and
out-degrees of a random digraph, in which case we can optionally forbid loops.
It can also be cast as a problem in random hypergraphs, or as a classical
occupancy, allocation, or coupon collection problem.

In each case, provided the two colour classes are not too different
in size nor the number of edges too low, we define a
probability space based on independent binomial variables and show that
its probability masses asymptotically equal those of the degrees
in the graph model almost everywhere.
The accuracy is sufficient to asymptotically determine the expectation
of any joint function of the degrees whose maximum is
at most polynomially greater than its expectation.
\end{abstract}

{\small \noindent\textbf{Keywords.} bipartite graph, degree sequence, random graph,
contiguity, digraph, directed graph, allocation, occupancy, coupon collection.}

\section{Introduction}\label{s:intro}
We will study the joint distributions of the vertex degrees for three different models of random bipartite graphs. In each case, we construct simpler probability spaces which match these distributions to high precision. The new probability spaces are based on independent binomial distributions and allow asymptotic calculations of any random variable which is a function of the degrees and has maximum at most polynomially greater than its expectation. In Section 2.1 we will show an example of such a calculation. Note that our results are much stronger than contiguity or decreasing total variation distance. These results are similar to those obtained by McKay and Wormald~\cite{degseq,degseq2} for the case of ordinary (not necessarily
bipartite) graphs.

We prefer to use graph terminology, but will also describe the
problem in the matrix and other settings.  Consider a probability
space of $m\times n$ matrices over $\{0,1\}$.  Three probability spaces will
be considered.  In the first case, which we call $\G_p$, some number $p\in(0,1)$ is specified
and each entry of the matrix is independently equal to 1 with
probability $p$ and equal to 0 otherwise.  In the second case, which we call $\G_k$, some integer $k$ is specified, and all $m\times n$ binary matrices
with exactly $k$ ones have the same probability, and no other matrices are allowed. In the third case, which we call $\G_\tvec$,
a list of $n$ integers $t_1,\ldots,t_n$ is specified, and all
$m\times n$ binary matrices with column sums $t_1,\ldots,t_n$,
respectively, are equally likely and no others are allowed.

We can interpret the matrix as a bipartite graph in the standard fashion.
Associate distinct vertices
$U=\{u_1,\ldots,u_m\}$ with the rows, and $V=\{v_1,\ldots,v_n\}$
with the columns, and place an edge between $u_i$ and $v_j$
exactly when the matrix entry in position $(i,j)$ equals~1.
The row and column sums of the matrix correspond to
the degrees of the vertices.

These probability models have also appeared in other settings. Given $m$ bins, at 
each stage $j=1, \ldots, n$ throw $t_j$ balls into distinct bins with all $\binom{m}{t_j}$ 
possible placings equally likely. Then the distribution of the number of balls in each 
bin $\Svec=(S_1, \ldots, S_m)$ can be studied.
This model is referred to as 
\emph{allocation by complexes} and is precisely our $\G_\tvec$ model.
If we allow the number of balls thrown to be a random variable $T_j$, binomially 
distributed with parameters $(m, p)$, we attain the $\G_p$ model. 

Similarly, in the \emph{coupon collection} problem a customer repeatedly buys a 
random number, $T$, of distinct coupons from a set of $m$ possible different coupons. 
This covers both our $\G_p$ case when $T$ is binomially distributed with parameters 
$(m, p)$ and our $\G_\tvec$ case where $T_j = t_j$ with probability 1. 
(Here, our vector $\svec$ describes the number of each coupon collected 
and $\tvec$ the number of coupons collected at each stage.)  

Finally, consider a hypergraph on $m$ vertices. At each stage $j=1, \ldots, n$,
choose at random a hyperedge of size $t_j$, allowing multi-edges. Then if we 
set $S_i$ to be the number of hyperedges which contain the $i$th vertex, we 
obtain the $\G_\tvec$ model.

If $m=n$, we can also associate the matrix with a directed graph.
There are $n$ vertices $\{w_1,\ldots,w_n\}$. A matrix entry
equal to~1 in position $(i,j)$ corresponds to a directed edge
from $w_i$ to $w_j$.  The case $i=j$ is permitted, so these directed
graphs can have loops. The row and column sums of the matrix
correspond to the out-degrees and in-degrees, respectively, of
the directed graph.
We will also treat the case of loop-free digraphs, which correspond
to square matrices with zero diagonal.  Our methods would also
work if some other limited set of matrix entries are required
to be zero, but we have not applied them in that case.

We now continue using the bipartite graph formulation.
For each of the three probability spaces of random bipartite graphs,
we seek to examine the $(m{+}n)$-dimensional joint distribution of the
vertex degrees.
If $G$ is a bipartite graph on $U\cup V$ (respecting the partition into
$U$ and $V$), then $\svec=\svec(G)=(s_1,\ldots,s_m)$
is the list of degrees of $u_1,\ldots,u_m$, and
$\tvec=\tvec(G)=(t_1,\ldots,t_n)$ is the list of degrees of
$v_1,\ldots,v_n$.
We call the pair $(\svec,\tvec)$ the \textit{degree sequence} of~$G$.

Define $I_n=\{0,1,\ldots,n\}$ and $I_{m,n}=I_n^m\times I_m^n$.
Also let $G(\svec,\tvec)$ be the number of (labelled) bipartite graphs on
$U\cup V$ with degree sequence $(\svec,\tvec)$.
In the case of $m=n$, we also define $\dig G(\svec,\tvec)$ to be
the number of loop-free digraphs with in-degrees $\svec$ and
out-degrees~$\tvec$.

For precision we need to distinguish between random variables (written
in uppercase) and the values they may take (written in lowercase).
For each probability space of random graphs, as determined by the context,
$\Svec=(S_1,\ldots,S_m)$ will denote the random variable given by the
degrees in $U$ and $\Tvec=(T_1,\ldots,T_n)$  will denote the random
variable given by the the degrees in~$V$.
We will take $\Svec$ to
have range~$I_n^m$ and $\Tvec$ to have range~$I_m^n$.
Also define random variables
\[
   K = \sum_{i=1}^m S_i \quad\text{and}\quad
   \varLambda = \frac{K}{mn}.
\]
As usual, $q$ is an abbreviation for $1-p$.

\subsection{Historical notes}\label{s:historical}

The $\G_\tvec$ model has received wide ranging attention, in particular the distribution of the number of isolated vertices. This is also a natural question in the alternative (non-graph) wordings of the model.
It corresponds to the number of empty bins in the allocation model \cite{BarHol89,HMP87,Kh97,MM2010,Vatutin82,Park81,Harris}, the number of uncollected coupons in the collector's problem \cite{Stadje1990, Mahmoud2010}, the number of isolated vertices in the hypergraph model and the number of zero rows in the binary matrix model~\cite{Holst}.
More generally, the number of vertices with a particular degree (or range of degrees) in $\G_\tvec$ has been studied in allocation \cite{Mitwalli2002,Mikhailov77a, Mikhailov77b}, graph \cite{AFH2008, Kordecki} and matrix models \cite{ESM72}.
A different extension on this theme is to study the distribution of the number of draws required to go from $i$ to $j$ non-empty bins \cite{AR2001,KJV2007,Mitwalli2002,Sellke1995,Smythe2011}.
In a similar direction, Khakimullin and Enatskaya studied the distribution of the number of draws to exceed a particular lineup in the bins in the $\G_\tvec$ model \cite{EKh1993} and in the i.i.d.~case which includes the $\G_p$ model as well~\cite{EKh2002}.
The monograph by Kolchin gives many results on $\G_\tvec$ phrased as the balls and bins model~\cite{KSCbook}.

We are interested in asymptotic results as we take $m,n$ roughly equal as they tend to infinity, but another natural option is to fix $m$, the number of vertices in one part, and let $n$, the number of vertices in the other part, tend to infinity.
There seems to be a consistent divide in the literature that when considered as a graph the asymptotics of $\G_\tvec$ are studied with $m, n$ both tending towards infinity while the balls and bins and coupon collection articles (including those cited above) fix $m$ and take $n$ tending toward infinity.
The latter corresponds to fixing the number of bins and taking the number of balls
to infinity or having a fixed number of coupons and letting the number of sampling
rounds tend to infinity.

In the other two probability models on bipartite graphs, $\G_p$ and $\G_k$, two types of results are known: those on the minimum and maximum degrees \cite{AFH2008, Bukor, Palka87Extreme} and those on the number of vertices with a given degree \cite{Kordecki, Palka83, Palka84}.
For results in the digraph counterpart $\digG_p$ see \cite{PalkaSperling} (and below). The model $\G_p$ also appears in papers on ball and bin models. Sometimes the numbers of balls thrown at each stage are allowed to be i.i.d.~random variables \cite{Kh05}.
If we then set these random variables to be binomially distributed with parameters $m, p$ we recover the $\G_p$ model. Godbole et.~al.~\cite{GLS99} study the number of sets of $r$ mutually threatening rooks.
This corresponds to the number of vertices with $h \geq r$ weighted by $\binom{h}{r}$ in our $\G_p$ and $\G_k$ models. 

Of the papers cited, we highlight some which concern the minimum and maximum degrees, a fixed number of the smallest and largest degrees and the  distribution of the $h^{\mathrm{th}}$ largest degree.

Khakimullin determined the asymptotic distribution of the $h^{\mathrm{th}}$ largest degree when the average degree increases faster than $\log m$ \cite{Kh05}.
The model used allowed the numbers of balls allocated at each step to be i.i.d~random variables and so includes both our $\G_p$ and uniform $\G_\tvec$ cases.
This extends an earlier result by the same author which gave the asymptotic distribution of the largest degree~\cite{Kh81}.
 
Palka and Sperling showed that if we fix $p$ such that $np = w(n) \log n = o(n)$, then any fixed number of the smallest and largest degrees are unique in $\digG_p$ and in the uniform $\G_\tvec$ model \cite{PalkaSperling}.
A similar result for the $\digG_\tvec$ model is shown by Palka in \cite{Palka86}, where $\tvec = (d,d,\ldots,d)$ and $d = w(n) \log n = o(n)$.
There is also some work on the degrees in random digraphs by Jaworski and Karo{\'n}ski \cite{JawKar} who showed, in the case that $\tvec=(d,d,\ldots, d)$ and $d=o(n)$, that the minimum vertex degree in $\G_\tvec$ is almost surely the same as that in $\digG_\tvec$.

\nicebreak
\subsection{Asymptotic notation}\label{s:asymptotic}

As we are dealing with asymptotics of functions of many variables,
we must be careful to define our asymptotic notation.

We will tacitly assume that all variables not declared to be constant
are functions of a single underlying index $\ell$ that takes values $1,2,\ldots\,$, and that all
asymptotic statements refer to $\ell\to\infty$.  Thus, the size
parameters $m,n$ are in reality functions $m(\ell)$ and $n(\ell)$,
and a statement like $f(m,n)=O(g(m,n))$ means that there is
a constant $A>0$ such that
$\abs{f(m,n)}\le A\abs{g(m,n)}$ when $\ell$ is large enough.
This should not be cause for alarm, because we will invariably 
impose conditions implying that $m,n\to\infty$ as $\ell\to\infty$.

The expression $\oo(1)$ represents any function of $\ell$ of 
magnitude $O(e^{-n^c})$ for some constant $c>0$.
The constant~$c$ might
be different for different appearances of the notation.
The class $\oo(1)$ is closed under addition, multiplication,
taking positive powers, and multiplication by polynomials
in~$n$.
 
\subsection{Graph models}\label{s:graphmodels}

We now define a sequence of finite probability spaces that we call
``models'', with sample space either $I(m,n)=I^m_n\times I^n_m$
or $I^m_n$. 
% In all cases the event space is the power set of the sample space.
The probability
measure for each model will be defined using random variables $(\Svec,\Tvec)$
or $\Svec$, respectively, whose distribution equals the respective
probability measure.  In general our notation will not distinguish between
each probability space and its probability measure.

We first consider six models whose probability measures are
derived from the degrees of a random bipartite graph or digraph~$G$.

\begin{enumerate}
  \item (\textit{$p$-models $\G_p,\digG_p$, for $0<p<1$})
    Generate $G$ by choosing each of the $mn$
    possible edges $u_iv_j$ with probability~$p$, such choices being independent.
    The probability distribution $\G_p=\G_p(m,n)$ on $I_{m,n}$ is that of the
    degree sequence $(\Svec,\Tvec)$ of~$G$. 
    If $m=n$ and 
    the edges $\{u_iv_i\}$ are forbidden, we obtain the probability distribution
    $\digG_p$ instead, corresponding to the degree sequences of a 
    loop-free digraph where each possible directed edge is chosen independently
    with probability~$p$. Note that $G(\svec,\tvec)=0$ for many pairs $(\svec,\tvec)$.
    We have
    \begin{align*}
      \Prob_{\G_p}(\stevent) &= p^k q^{mn-k} G(\svec,\tvec), \\
      \Prob_{\digG_p}(\stevent) &= p^k q^{n^2-n-k} \dig G(\svec,\tvec).
    \end{align*}
    where $q=1-p$ and $k=\sum_{i=1}^m s_i$.
  \item (\textit{$k$-models $\G_k, \digG_k$, for integer $k\ge 0$})
    Generate $G$ by choosing each of the bipartite graphs
    on $U\cup V$ having $k$ edges, with equal probability.
    The probability distribution $\G_k=\G_k(m,n)$ on $I_{m,n}$ is that of the
    degree sequence $(\Svec,\Tvec)$ of~$G$. 
    If $m=n$ and the edges $\{u_iv_i\}$ are forbidden, we obtain the 
    distribution $\digG_k=\digG_k(n)$ of the degree-sequences for
    the uniform probability space of all loop-free digraphs with $k$ edges. 
    We have
    \begin{align*}
      \Prob_{\G_k}(&\stevent) \\
      &=   \begin{cases}
                   \displaystyle \binom{mn}{k}^{\!\!-1} G(\svec,\tvec),
                      & \text{if~}\sum_{i=1}^m s_i=\sum_{j=1}^n t_j=k;\\[1ex]
                   0, & \text{otherwise},
          \end{cases} \displaybreak[0] \\
    \Prob_{\digG_k}(&\stevent) \\
    &=    \begin{cases}
              \displaystyle \binom{n^2-n}{k}^{\!\!-1} \dig G(\svec,\tvec),
                & \text{if~}\sum_{i=1}^n s_i=\sum_{j=1}^n t_j=k;\\[1ex]
             0, & \text{otherwise},
          \end{cases}
    \end{align*}
  \item (\textit{$\tvec$-models $\G_\tvec,\digG_\tvec$, for $\tvec\in I_m^n$})
  Generate $G$ by choosing each of the bipartite graphs
    on $U\cup V$ having $\tvec(G)=\tvec$, with equal probability.
    {For consistency we can define the random variable $\Tvec$ to
    have the value $\tvec$, but since this is constant we will
    define our probability spaces using $\Svec$ only.}
    The probability distribution $\G_\tvec=\G_\tvec(m)$ on $I_n^m$ is that of the
    degree sequence $\Svec$ of~$G$ in~$U$. 
    If $m=n$ and 
    the edges $\{u_iv_i\}$ are forbidden, we obtain the distribution
     $\digG_\tvec=\digG_\tvec(n)$ of the in-degrees for the uniform probability
    distribution of all loop-free digraphs with fixed out-degrees~$\tvec$.
    For a given $\tvec\in I_m^n$, we have
    \begin{align*}
      \Prob_{\G_\tvec}(\Svec=\svec) &=
          \tprod^{\!\!-1} G(\svec,\tvec), \\
      \Prob_{\digG_\tvec}(\Svec=\svec) &=
        \dtprod^{\!\!-1} \dig G(\svec,\tvec).
   \end{align*} 
\end{enumerate}

The probability spaces $\G_p$, $\G_k$ and $\G_\tvec$ are clearly related,
by mixing and conditioning.
In particular, for any event
$E \subseteq I_{m,n} $ or $E'\subset I_n^m$, the following hold.
Note that the first relationships on lines~\eqref{relM} and~\eqref{relt} are
independent of $p$ and assume $0<p<1$.
\begin{align}
   \Prob_{\G_p}(E) &= \sum_{k=0}^{mn} \binom{mn}{k} p^k q^{mn-k} \Prob_{\G_k}(E)
           = \sum_{\tvec\in I_m^n}\,
             \biggl(\, \tprod p^{t_j} q^{m-t_j} \biggr) \Prob_{\G_\tvec}(E),
                       \label{relp} \displaybreak[0] \\
   \Prob_{\G_k}(E) &= \Prob_{\G_p}(E \, \bigr| \, {K=k})\,
         = \sum_{\tvec\suchthat\sum_{j=1}^n t_j=k} \binom{mn}{k}^{\!\!-1}
            \tprod \Prob_{\G_\tvec}(E), \label{relM}\\
  { \Prob_{\G_\tvec}(E') }
    &= {\Prob_{\G_p}(E'\times\{\tvec\} \, \bigr| \, {\Tvec=\tvec}) 
     = \Prob_{\G_{k=\sum_{j}t_j } }(E'\times\{\tvec\} \, \bigr| \,{\Tvec=\tvec}),} \label{relt}
\end{align}
with similar relations between $\digG_p$, $\digG_k$ and $\digG_\tvec$.

Note that the separate distributions of $\Svec$ and $\Tvec$ in
$\G_p$ and $\G_k$ are elementary.  In~$\G_p$, the components
of $\Svec$ have independent binomial distributions, while in
the $\G_k$ model
$\Svec$ has a multivariate hypergeometric distribution.
The difficulty is in quantifying the dependence between
$\Svec$ and $\Tvec$ when all $m+n$ components are 
considered together.

\subsection{Binomial models}\label{s:binomialmodels}

Our aim is to compare the degree sequence distributions defined above
to some distributions derived from independent binomials.  Our motivating
observation is the known marginal distributions of $\Svec$ and $\Tvec$
in the models $\G_p$ and $\G_k$.

\begin{enumerate}
  \item (\textit{Independent models\/ $\I_p,\digI_p$, for $0<p<1$})
  Generate $m$ components distributed $\Bin(n,p)$ and $n$ components
   distributed $\Bin(m,p)$, all $m+n$ components being independent.
   The joint distribution on $I_{m,n}$ is $\I_p=\I_p(m,n)$.
   If instead we have $m=n$ and the $2n$ components are all distributed
   $\Bin(n{-}1,p)$, the
   joint distribution on $I_{n,n}$ is $\digI_p=\digI_p(n)$. We have
    \begin{align*}
      \Prob_{\I_p}(&\stevent) \\
       &= p^{\sum_i s_i+\sum_j t_j} q^{2mn-\sum_i s_i-\sum_j t_j} \sprod\tprod, \\
    \Prob_{\digI_p}(&\stevent) \\
  &=    p^{\sum_i s_i+\sum_j t_j} q^{2n^2-2n-\sum_i s_i-\sum_j t_j} \dsprod\dtprod. 
    \end{align*}
   \item (\textit{Binomial $p$-models $\B_p,\digB_p$, for $0<p<1$})
   The distribution $\B_p=\B_p(m,n)$ on $I_{m,n}$
   is the conditional distribution of
   $\I_p$ subject to $\sum_{i=1}^m S_i=\sum_{j=1}^n T_j$.
   For $m=n$,
   the distribution $\digB_p=\digB_p(n)$ on $I_{n,n}$ is obtained
   from $\digI_p$ by the same conditioning.
   We have
    \begin{align*}
      \Prob_{\B_p}&(\stevent)\\  
      &=
          \begin{cases}
                   \displaystyle \frac
                      {\Prob_{\I_p}(\stevent)}
                      {\Prob_{\I_p}\(\sum_{i=1}^m S_i=\sum_{j=1}^n T_j\)},
                      & \text{if~}\sum_{i=1}^m s_i=\sum_{j=1}^n t_j;\\[1ex]
                   0, & \text{otherwise},
          \end{cases}
    \end{align*}
    and similarly for $\digB_p$.
   \item (\textit{Binomial $k$-models $\B_k,\digB_k$, for integer {$k\ge0$}})
   The distribution $\B_k=\B_k(m,n)$ on $I_{m,n}$
   is the conditional distribution of
   $\I_p$ subject to $\sum_{i=1}^m S_i=\sum_{j=1}^n T_j = k$.
   For $m=n$, $\digB_k=\digB_k({m,n})$ is derived from $\digI_p$ in the
   same way.  In both cases, the distribution doesn't depend on~$p$.
   We have
    \begin{align*}
     & \Prob_{\B_k}(\Svec=\svec\wedge\Tvec=\tvec)\\
          &{\quad}=\begin{cases}
                   \displaystyle \binom{mn}{k}^{\!\!-2}\sprod\tprod,
                      & \text{if~}\sum_{i=1}^m s_i=\sum_{j=1}^n t_j=k;\\[1ex]
                   0, & \text{otherwise},
          \end{cases} \\
       &\Prob_{\digB_k}(\Svec=\svec\wedge\Tvec=\tvec)\\
             &{\quad}=\begin{cases}
                     \displaystyle \binom{n^2-n}{k}^{\!\!-2}\dsprod\dtprod,
                      & \text{if~}\sum_{i=1}^n s_i=\sum_{j=1}^n t_j=k;\\[1ex]
                   0, & \text{otherwise}.
               \end{cases}
    \end{align*}
    In each case $\Svec$ and $\Tvec$ have independent multivariate
    hypergeometric distributions.
   \item (\textit{Binomial $\tvec$-models $\B_\tvec,\digB_\tvec$, for $\tvec\in I_m^n$})
     The distribution $\B_\tvec=\B_\tvec(m,n)$ on $I_n^m$ is the
     distribution of $\Svec$ when $(\Svec,\Tvec)$ has distribution
     $\B_k$ for $k=\sum_{j=1}^n t_j$.
     For $m=n$, $\digB_\tvec=\digB_\tvec(n)$ is derived from 
     $\digB_k$ in the same way.
     For a given $\tvec\in I_m^n$, we have
     \begin{align*}
      \Prob_{\B_\tvec}(\Svec=\svec)&=\begin{cases}
                    \displaystyle \binom{mn}{k}^{\!\!-1} \sprod
                      & \text{if~}\sum_{i=1}^m s_i=\sum_{j=1}^n t_j;\\[1ex]
                   0, & \text{otherwise},
          \end{cases} \\
     \Prob_{\digB_\tvec}(\Svec=\svec)&=\begin{cases}
                     \displaystyle \binom{n^2-n}{k}^{\!\!-1} \dsprod
                      & \text{if~}\sum_{i=1}^n s_i=\sum_{j=1}^n t_j;\\[1ex]
                   0, & \text{otherwise},
               \end{cases}
    \end{align*}
    In each case, $\Svec$ has a multivariate hypergeometric distribution.
  \item (\textit{Integrated $p$-models $\V_p,\digV_p$, for $0<p<1$})
    The distribution $\V_p=\V_p(m,n)$ on $I_{m,n}$ is a mixture of
    $\B_{p'}$ distributions, while for $m=n$ the distribution
    $\digV_p=\digV_p(n)$ on $I_{n,n}$ is a mixture of
    $\digB_{p'}$ distributions.
    Let
    \begin{align*}
        K_p(p') &= \biggl(\frac {mn}{\pi pq}\biggr)^{\!1/2}
             \exp\biggl( -\frac{mn}{pq} (p'-p)^2\biggr), \\[0.5ex]
        V(p) &= \int_0^1 \!K_p(p')\,dp'.
    \end{align*}    
    Then we define 
    \begin{align*}
     \Prob_{\V_p} (\stevent) = V(p)^{-1} \int_0^1 
        \! K_p(p') \Prob_{\B_{p'}} (\stevent)\,dp', \\
     \Prob_{\digV_p} (\stevent) = V(p)^{-1} \int_0^1 
        \! K_p(p') \Prob_{\digB_{p'}} (\stevent)\,dp'.
    \end{align*}
\end{enumerate}

Our main theorems will show that, under certain conditions, $\G_p$
is very close to $\V_p$, $\G_k$ to $\B_k$, and $\G_\tvec$ to $\B_\tvec$.
Similar relationships hold for the digraph models.

\subsection{The main theorems}\label{s:main}

Consider positive integers $m,n$ and real variable $x\in(0,1)$.
(As mentioned in Section~\ref{s:asymptotic}, these variables
are actually functions of a background index~$\ell$.)
For constants $a,\eps>0$, we say that $(m,n,x)$ is 
\textit{$(a,\eps)$-acceptable} if
\begin{gather}
  m,n\to\infty \text{ with } m=o(n^{1+\eps}), n=o(m^{1+\eps}), \text{ and }
    \notag \\
 \frac{(1-2x)^2}{4x(1-x)}
      \biggl( 1 + \frac{5m}{6n} + \frac{5n}{6m} \biggr)<a\log n.\label{C2}
\end{gather}
Note that~\eqref{C2} implies $x(1-x)=\Omega\((\log n)^{-1}\)$.

For $\eps>0$, a vector $(x_1,x_2,\ldots,x_N)$ will be called
\textit{$\eps$-regular} if
\[
   x_i-\frac1N\sum_{j=1}^N x_j=O(N^{1/2+\eps})
\]
uniformly for $i=1,\ldots,N$.
We say that $(\svec,\tvec)$ is \textit{$\eps$-regular} if
$\sum_{i=1}^m s_i=\sum_{j=1}^n t_j$ and
$\svec, \tvec$ are both $\eps$-regular.

Finally, define $\lambda_m(\tvec) = (mn)^{-1}\sum_{j=1}^n t_j$.
If $\sum_{i=1}^m s_i=\sum_{j=1}^n t_j$, the common value of
$\lambda_n(\svec)$ and $\lambda_m(\tvec)$ will be denoted by~$\lambda$.
Note that $\lambda$ is the value in $[0,1]$ that gives the density
of a bipartite graph with degrees $(\svec,\tvec)$, relative to $K_{m,n}$.
In the case of loop-free digraphs, $\lambda\in[0,1-1/n]$.

We now state the theorems that are the main contribution of this
paper.  Their proofs will be given in Section~\ref{s:proofs}, after
some preliminary lemmas are proved in Section~\ref{s:properties}.

\begin{thm}\label{twosides}
 Let constants $a,b>0$ satisfy $a+b<\tfrac12$.
 Then there is a constant $\eps=\eps(a,b)>0$ such that the
   following holds. Let $\D$ and $\D'$ be probability
   spaces on $I_{m,n}$ in one of the following cases.
  \begin{enumerate} \itemsep=0pt
    \item[(a)] $(m,n,p)$ is $(a,\eps)$-acceptable and $(\D,\D')=(\G_p,\V_p)$,
    \item[(b)] $m=n$, $(n,n,p)$ is $(a,\eps)$-acceptable and $(\D,\D')=(\digG_p,\digV_p)$,
    \item[(c)] $(m,n,k/mn)$ is $(a,\eps)$-acceptable and $(\D,\D')=(\G_k,\B_k)$,
    \item[(d)] $m=n$, $(n,n,k/n^2)$ is $(a,\eps)$-acceptable and $(\D,\D')=(\digG_k,\digB_k)$,
  \end{enumerate}
  Then there is an event $B=B(\D)\subseteq I_{m,n}$
 such that $\Prob_{\D}(B) = \oo(1)$,
 and uniformly for $(\svec,\tvec)\in I_{m,n}\setminus B$,
 \[
    \Prob_{\D}(\stevent) =
    \(1 + O(n^{-b})\) \Prob_{\D'}(\stevent).
 \]
 Moreover, let $X:I_{m,n} \to \Reals$ be a random variable and
 let $E\subseteq I_{m,n}$ be an event. Then,
 \begin{align*}
   \Prob_{\D}(E) &=
    \(1 + O(n^{-b})\) \Prob_{\D'}(E) + \oo(1), \\[1ex]
   \expect_{\D}(X) &= \expect_{\D'}(X)
            + O(n^{-b}) \expect_{\D'}(\abs X)
            + \oo(1) \max_{(\svec,\tvec)\in I_{m,n}} \abs X,\\
   \Var_{\D}(X) &= \(1+O(n^{-b})\) \Var_{\D'}(X)
	    + \oo(1) \max_{(\svec,\tvec)\in I_{m,n}} X^2.
 \end{align*} 
\end{thm}

\begin{cor}\label{vanishing}
  Let $E\subseteq I_{m,n}$ be an event.  Then, under
  the conditions of Theorem~\ref{twosides}, 
  \begin{align*}
   \text{if } \Prob_{\B_p}(E) &\to 0 \text{ then }
    \Prob_{\G_p}(E) = \oo(1) + o(1)\ssqrt{\Prob_{\B_p}(E)}\,, \text{ and} \\ 
   \text{if } \Prob_{\G_p}(E) &\to 0 \text{ then }
    \Prob_{\B_p}(E) = \oo(1) + o(1)\ssqrt{\Prob_{\G_p}(E)}\,. \\
  \intertext{Similarly, for $m=n$,}
  \text{if }\Prob_{\digB_p}(E) &\to 0 \text{ then }
      \Prob_{\digG_p}(E) = \oo(1) + o(1)\ssqrt{\Prob_{\digB_p}(E)}\,,  \text{ and} \\
 \text{if }\Prob_{\digG_p}(E) &\to 0 \text{ then }
      \Prob_{\digB_p}(E) = \oo(1) + o(1)\ssqrt{\Prob_{\digG_p}(E)}\,.
  \end{align*}
  In particular, $\G_p$ and $\B_p$ are contiguous; i.e.,
   $\Prob_{\G_p}(E) \to 0$ if and only if $\Prob_{\B_p}(E) \to 0$,
   and similarly for $\digG_p$ and $\digB_p$.
\end{cor}

\begin{thm}\label{oneside}
 Let constants $a,b>0$ satisfy $a+b<\tfrac12$.
 Then there is a constant $\eps=\eps(a,b)>0$ such that the
   following holds whenever   
   $(m,n,\lambda_m(\tvec))$ is $(a,\eps)$-acceptable and
    $\tvec$ is $\eps$-regular.
    Let $\D$ and $\D'$ be probability
   spaces on $I_n^m$ in one of the following cases.
  \begin{enumerate} \itemsep=0pt
    \item[(a)]  $(\D,\D')=(\G_\tvec,\B_\tvec)$,
    \item[(b)] $m=n$ and $(\D,\D')=(\digG_\tvec,\digB_\tvec)$.
  \end{enumerate}
  Then there is an event $B=B(\D)\subseteq I_n^m$
 such that $\Prob_{\D}(B) = \oo(1)$,
 and uniformly for $\svec\in I_n^m\setminus B$,
 \[
    \Prob_{\D}(\Svec=\svec) =
    \(1 + O(n^{-b})\) \Prob_{\D'}(\Svec=\svec).
 \]
 Moreover, let $X:I_n^m \to \Reals$ be a random variable and
 let $E\subseteq I_n^m$ be an event. Then,
 \begin{align*}
   \Prob_{\D}(E) &=
    \(1 + O(n^{-b})\) \Prob_{\D'}(E) + \oo(1), \\[1ex]
   \expect_{\D}(X) &= \expect_{\D'}(X)
            + O(n^{-b}) \expect_{\D'}(\abs X)
            + \oo(1) \max_{\svec\in I_n^m}\, \abs X,\\
   \Var_{\D}(X) &= \(1+O(n^{-b})\) \Var_{\D'}(X)
	    + \oo(1) \max_{\svec\in I_n^m} X^2.
 \end{align*} 
\end{thm}

A weak corollary of these theorems is that each of the distribution
pairs $(\G_p,\V_p)$, $(\digG_p,\digV_p)$, $(\G_k,\B_k)$,
$(\digG_k,\digB_k)$, $(\G_\tvec,\B_\tvec)$ and $(\digG_\tvec,\digB_\tvec)$
have total variation distance $O(n^{-b})$ under the stated conditions.

The proofs of the theorems will be presented in
Sections~\ref{s:properties} and~\ref{s:proofs}.  Meanwhile,
we will give an example that illustrates how the theorems can
be applied.

\nicebreak
\section{Some useful lemmas and an example}\label{s:example}

We first record a few elementary properties.

\begin{lemma}\label{Bpdenom}
If\/ $\sum_{i=1}^m s_i = \sum_{j=1}^n t_j = k$ and $pqmn\to\infty$, then
\begin{align*}
  &\Prob_{\B_p}(\stevent) \\
  &{\quad}=      \(2+O((pqmn)^{-1})\)\, p^{2k} q^{2mn-2k} \sqrt{\pi pqmn}\,
        \sprod\tprod, \displaybreak[0] \\
  &\Prob_{\digB_p}(\stevent) \\
       &{\quad}= \(2+O((pqn^2))^{-1})\)\, p^{2k} q^{2n^2-2n-2k} \sqrt{\pi pqn(n{-}1)}\,
       \dsprod\dtprod.
 \end{align*}
 uniformly over $\svec,\tvec$.
\end{lemma}
\begin{proof}
  In $\I_p$, both $\sum_{i=1}^m S_i$ and $\sum_{j=1}^n T_j$ have the
   distribution $\Bin(mn,p)$.  Therefore
  \begin{align}
    \Prob_{\I_p}\Bigl({\textstyle \sum_{i=1}^m S_i=\sum_{j=1}^n T_j}\Bigr)
    &= \sum_{k=0}^{mn} \binom{mn}{k}^{\!2} p^{2k}q^{2mn-2k} \label{sumsq}\\
    &= \frac{1}{2\sqrt{\pi pqmn}}\(1+O((pqmn)^{-1})\).\notag
\end{align}
For the last step we use that the central part of the sum is
approximately normal and sum it with the Euler-Maclaurin formula, while
the two tails of the sum are negligible in comparison.
The first claim now
 follows from the formulas for $\Prob_{\B_p}(\stevent)$
 and $\Prob_{\I_p}(\stevent)$.
 The second claim is proved in the same manner.
\end{proof}

\begin{lemma}\label{Vpest}
If $pqmn\to\infty$, then
\[
   V(p) = 1 - o\(e^{-pqmn}\).
\]
\end{lemma}
\begin{proof}
 $K_p(p')$ is a normal density with mean $p$ and variance $pq/(2mn)$,
 so we just need to apply standard normal tail bounds to the definition of $V(p)$.
\end{proof}

The next lemma demonstrates how statistics of variables in $\B_p$
can be converted into statistics in $\V_p$.
Note that $X$ can be the indicator variable of an event, so the
lemma applies to probabilities as well.

\begin{lemma}[{\cite{degseq}}]\label{BtoV}
Let $X$ be a random variable on $I_{m,n}$.  Then
\begin{align*}
  \expect_{\V_p}(X) &= V(p)^{-1} \int_0^1 \! K_p(p')\expect_{\B_{p'}}(X)\,dp',\\
  \Var_{\V_p}(X) &= V(p)^{-1} \int_0^1 
        \! K_p(p')\( \Var_{\B_{p'}}(X) 
             + (\expect_{\V_{p}}(X)-\expect_{\B_{p'}}(X))^2 \)\,dp'.
\end{align*}
\end{lemma}

\subsection{Vertices of low degree in random digraphs}

We now provide an example of how Theorem~\ref{twosides}(b)
can be applied to random digraphs.  Since this is only an illustration,
we will not attempt to treat all values of the parameters or to obtain
the best possible error terms.

Let $G$ be a random loop-free digraph on $n$ vertices and edge
probability $p=\tfrac12$.
For convenience we will assume that $n$ is even, though treatment
of the odd case would be much the same.
As usual, $S_1,\ldots,S_n$ are the out-degrees of the vertices, and
$T_1,\ldots,T_n$ are the in-degrees.
Let $X,Y$ be random variables which count the vertices
with out-degree at most $\tfrac n2-1$, and the vertices with
in-degree at most $\tfrac n2-1$, respectively.
It is easy to see that each of $X$ and $Y$ has a distribution
exactly $\Bin(n,\tfrac12)$, but that $X$ and $Y$ are not
independent.
Our aim will be to find their asymptotic joint distribution.

We will first calculate some properties of binomial distributions truncated
at the centre. Application of model $\digV_{1/2}$ requires us to consider
probabilities close to~$\tfrac12$.

\begin{lemma}\label{truncated}
 The following hold when $\eps>0$ is sufficiently small.
 Let $n$ be even and
 let $p=\tfrac12+\delta$ where $\delta=O(n^{-1+\eps})$.
 For $0\le k\le \frac n2-1$ define
 \[ b(p,k) = \binom{n-1}{k}p^k(1-p)^{n-1-k} \]
 and $P(\delta) = \sum_{k=0}^{n/2-1} b(p,k)$.
 Then
 \begin{equation}\label{Pdelta}
    P(\delta) = \dfrac12 - \delta\,\sqrt{\frac{2n}{\pi}}
         + O(n^{-1-\eps}).
 \end{equation}
 Now define two random variables, $Z^-_\delta$ by truncating $\Bin(n-1,p)$
 to $[0,\tfrac12 n-1]$, and $Z^+_\delta$ by truncating $\Bin(n-1,p)$
  to $[\tfrac12 n,n-1]$.
 Then
 \begin{align}
    \expect(Z^-_\delta) &= \dfrac12 n - \sqrt{\frac{n}{2\pi}} + O(n^{1/2-\eps}),
     &  
     \expect(Z^+_\delta) &= \dfrac12 n + \sqrt{\frac{n}{2\pi}} + O(n^{1/2-\eps}),\\
    \Var(Z^-_\delta) &= \frac{(\pi-2)n}{4\pi} + O(n^{1-\eps}),
    &
    \Var(Z^+_\delta) &= \frac{(\pi-2)n}{4\pi} + O(n^{1-\eps}).
 \end{align}
\end{lemma}
\begin{proof}
  Define $s_j=b(p,\frac12 n-1-j)$.  From~\cite{deghalfn} we have
  for $j=O(n^{1/2+\eps})$ that
  \begin{align*}
     s_0 &= \(1 - 2\delta - 2\delta^2 n - 1/(4n) + O(n^{-1-\eps})\)\,\sqrt{\frac{2}{\pi n}},
       \text{ and} \\
     \frac{s_j}{s_0} &= \( 1 - 4j^4/(3n^3) + O(n^{-1-\eps})\)
        \exp\( -4\delta j - 2j(j+1)/n\).
  \end{align*}
  By summing $s_j/s_0$ using the Euler-Maclaurin method, as in~\cite{deghalfn},
  we obtain the formula for $P(\delta)$.
  Similarly summing $js_j/s_0$ and $j^2s_j/s_0$, we obtain the formulas
  for $\expect(Z^-_\delta)$ and $\Var(Z^-_\delta)$.

  Finally, note that the truncations divide the range exactly in half, and
  so we have $1-P(\delta)=P(-\delta)$, $\expect(Z^-_\delta)+\expect(Z^+_{-\delta})=n-1$
  and $\Var(Z^-_\delta)=\Var(Z^+_{-\delta})$.
  This proves the statistics for $Z^+_\delta$.
\end{proof}

\begin{thm}\label{digraphs}
  Suppose $n$ is even and $x,y$ are integers with $x,y=O(n^{-1/2+\eps})$ 
  for sufficiently small $\eps>0$.  Then
  \[
    \Prob_{\digG_{1/2}}\( (X=\tfrac12 n+x)\wedge (Y=\tfrac12 n+y)\)
     = \frac{2+o(1)}{n\sqrt{\pi^2-4}}
       \exp\biggl( -\frac{2\pi(\pi x^2 + \pi y^2 - 4xy)}{(\pi^2-4)n} \biggr).
  \]
\end{thm}
\begin{proof}
Theorem~\ref{twosides}(b) tells us to calculate the probability in 
$\digV_{1/2}$, for which we need the probability in $\digB_p$
when $p\approx\tfrac12$.
For integers $x,y$, define events
\begin{align*}
  E(x,y) &= \bigl\{ (\Svec,\Tvec)
          \mid X=\tfrac12 n + x \wedge Y=\tfrac12 n + y\bigr\}, \text{ and} \\
  E_\Sigma &= \bigl\{ (\Svec,\Tvec) 
          \mid {\textstyle \sum_{i=0}^n S_i = \sum_{j=0}^n T_j} \bigr\}.
\end{align*}
Recall that $\digB_p$ is $\digI_p$ conditioned on event~$E_\Sigma$ so, applying 
Bayes' rule twice,
\begin{equation}\label{bayes}
  \Prob_{\digB_p} (E(x,y)) =
  \Prob_{\digI_p} (E(x,y))
    \, \frac{\Prob_{\digI_p}(E_\Sigma\mid E(x,y))}{\Prob_{\digI_p} (E_\Sigma)}.
\end{equation}
We have already computed $\Prob_{\digI_p} (E_\Sigma)$ in~\eqref{sumsq};
for $p=\tfrac12 + o(1)$ it is
\begin{equation}\label{part3}
  \Prob_{\digI_p} (E_\Sigma) = \frac{1+ o(1)}{n\sqrt\pi}.
\end{equation}
Now consider $\Prob_{\digI_p} (E_\Sigma\mid E(x,y))$.
Under this conditioning, symmetry implies that
$\sum_i S_i$ has the same distribution as the sum of $\tfrac12 n+x$
copies of $Z^-_\delta$ and $\tfrac12 n-x$ copies of $Z^+_\delta$,
all of these being independent.  A similar fact holds for $\sum_j T_j$,
which is in addition independent of $\sum_i S_i$ since we
are operating in~$\digI_p$.
Also recall that the binomial distribution and therefore its truncations and
their convolutions are
log-concave, so we know from~\cite{Bender} that
$\varDelta = \sum_i S_i - \sum_j T_j$, in $\digI_p$ conditioned on $E(x,y)$,
satisfies a local central-limit theorem.
Using Lemma~\ref{truncated}, we calculate
\begin{align*}
  \expect_{\digI_p}(\varDelta\mid E(x,y)) 
       &= (y-x)\,\sqrt{\frac{2n}{\pi}} + O(n^{1-\eps}) \\
  \Var_{\digI_p}(\varDelta\mid E(x,y)) 
      &= \frac{(\pi-2)n^2}{2\pi} + O(n^{2-\eps}),
\end{align*}
and so
\begin{equation}\label{part2}
  \Prob_{\digI_p} (E_\Sigma\mid E(x,y)) = \frac{1+ o(1)}{n\sqrt{\pi-2}}
   \exp\biggl( -\frac{2(x-y)^2}{(\pi-2)n} \biggr).
\end{equation}
Finally, consider $\Prob_{\digI_p} (E(x,y))$.  Since the $2n$ events
$S_i\le\tfrac12 n-1, T_j\le\tfrac12 n-1$ are independent in $\digI_p$,
$X$ and $Y$
are independent Binomial variables $\Bin(n,P(\delta))$.
Using~\eqref{Pdelta} and the normal approximation for the binomial
distribution, we have
\[
  \Prob_{\digI_p} (E(x,y)) =
    \frac{2+o(1)}{\pi n} \exp\biggl(
      -\frac{2(x\sqrt\pi+\delta\sqrt{2n^3}\,)^2}{\pi n}
      -\frac{2(y\sqrt\pi+\delta\sqrt{2n^3}\,)^2}{\pi n}
      \biggr).
\]
Applying this to~\eqref{bayes} together with~\eqref{part3} and~\eqref{part2},
we find that
\begin{align*}
  \Prob_{\digB_p} (E(x,y)) &= \frac{2+o(1)}{n\sqrt{\pi(\pi-2)}}\\
  &{\quad}\times
    \exp\biggl(
      -\frac{2(x\sqrt\pi+\delta\sqrt{2n^3}\,)^2}{\pi n}
      -\frac{2(y\sqrt\pi+\delta\sqrt{2n^3}\,)^2}{\pi n}
      -\frac{2(x-y)^2}{(\pi-2)n} \biggr).
\end{align*}

Now we apply Lemma~\ref{BtoV} to pass the result to $\digV_{1/2}$.
Multiplying by $K_{1/2}(\tfrac12+\delta)$ and integrating, we obtain the
formula in the theorem, which holds for $\digG_{1/2}$ on account
of Theorem~\ref{twosides}(b).
\end{proof}

A corollary of the theorem is that $X-Y$ and $X+Y$ have
asymptotically independent normal distributions, apart from
necessarily having the same parity.

\begin{cor}\label{sumdiff}
  Under the conditions of the theorem, let $\alpha,\beta$ be integers
  of the same parity such that $\alpha,\beta=O(n^{1/2+\eps})$.
  Then
  \[
     \Prob_{\digG_{1/2}}\( (X+Y=n+\alpha) \wedge (X-Y=\beta)\) 
       = \frac{2+o(1)}{n\sqrt{\pi^2-4)}}
    \exp\biggl( -\frac{\pi \alpha^2}{(\pi+2)n} - \frac{\pi \beta^2}{(\pi-2)n}
       \biggr).
  \]
\end{cor}

More complex information could also be obtained, such as the distributions
of all the order statistics of the degrees, but the calculations
would be considerably more intricate. See~\cite{degseq2} for similar
calculations for ordinary graphs.

\nicebreak
\section{Properties of likely degree sequences}\label{s:properties}

To prove our theorems,
our first task will be to investigate the bulk behaviour of
our various probability spaces in order to identify some
behaviour that has probability $\oo(1)$.
We will apply a few concentration inequalities, which we now give.

\begin{thm}[{\cite[Lemma 1.2]{mcdiarmid}}]\label{concentration}
 Let $\Xvec = (X_1,X_2,\ldots,X_N)$ be a family of independent random
 variables, with $X_i$ taking values in a set $A_i$ for each~$i$.
 Suppose that for each $j$ the function $f:\prod_{i=1}^N A_i\to\Reals$ satisfies
 $\abs{f(\xvec)-f(\xvec')}\le c_j$ whenever $\xvec,\xvec'\in\prod_{i=1}^N A_i$
 differ only in the $j$-th component.  Then, for any $z$,
 \[ \textstyle
   \Prob\(\Abs{f(\Xvec)-\expect(f(\Xvec))} \ge z\)
     \le 2\exp\(-2z^2/\sum_{i=1}^N c_i^2\).
 \]
\end{thm}
\begin{cor}\label{concentrationcor}
 Let $\Xvec = (X_1,X_2,\ldots,X_N)$ be a family of independent real 
 random
 variables such that $\abs{X_i-\expect(X_i)} \le c_i$ for each~$i$.
 Define $X = \sum_{i=1}^N X_i$.  Then, for any $z$,
 \[ \textstyle
    \Prob\(\abs{X-\expect(X)} \ge z\)
     \le 2\exp\(-\tfrac12 z^2/\sum_{i=1}^N c_i^2\,\)
 \]
\end{cor}

Another consequence of Theorem~\ref{concentration} is the following.

\begin{thm}\label{setfunction}
  Let $A_1,\ldots,A_N$ be finite sets, and let $a_1,\ldots,a_N$ be integers
  such that $0\le a_i\le\card{A_i}$ for each~$i$.
  Let $\binom{A_i}{a_i}$ denote the uniform probability space of
  $a_i$-element subsets of~$A_i$.
  Suppose that for each $j$ the function
  $f:\prod_{i=1}^N \binom{A_i}{a_i}\to\Reals$ satisfies
  $\abs{f(\xvec)-f(\xvec')}\le c_j$ whenever
  $\xvec,\xvec'\in\prod_{i=1}^N \binom{A_i}{a_i}$ are the same except that their
  $j$-th components $x_j,x'_j$ have $\card{x_j\cap x'_j}=a_j-1$
  (i.e., the $a_j$-element subsets $x_j,x'_j$ are minimally different).
  If $\Xvec=(X_1,\ldots,X_N)$ is a family of independent set-valued
  random variables with distributions 
  $\binom{A_1}{a_1},\ldots,\binom{A_N}{a_N}$, then for any $z$,
  \[ 
   \Prob\(\Abs{f(\Xvec)-\expect(f(\Xvec))} \ge z\)
     \le 2\exp\biggl(\frac{-2z^2}
           {\sum_{i=1}^N c_i^2\min\{a_i,\card{A_i}-a_i\}}\biggr).
 \]
\end{thm}
\begin{proof}
  We start by reminding the reader of a classical algorithm called
  ``reservoir sampling'', attributed by Knuth to Alan G. Waterman~\cite[p.\,144]{Knuth}.
  Let $Y^{(i)}_{a_i+1},\ldots,Y^{(i)}_{\card{A_i}}$ be independent random
  variables, where $Y^{(i)}_{j}$ has the discrete uniform distribution
  on $\{1,2,\ldots,j\}$.  Now suppose $A_i=\{w_1,\ldots,w_{\card{A_i}}\}$.
  Execute the following algorithm:
  \begin{tabbing} 
      \quad\=For $j=1,\ldots,a_i$ set $x_j := w_j$\,;\\
      \>For $j=a_i+1,\ldots,\card{A_i}$,
        if $Y^{(i)}_{j} \le a_i$ then set $x_{Y^{(i)}_{j}} := w_j$\,.
  \end{tabbing}
  Define $X_i=X_i(Y^{(i)}_{a_i+1},\ldots,Y^{(i)}_{\card{A_i}})$ to be the
  value of $\{x_1,\ldots,x_{a_i}\}$ when the algorithm finishes.
  The \textit{raison d'\^ etre\/} of the algorithm, which is easy to check,
  is that $X_i$ has distribution $\binom{A_i}{a_i}$; i.e., it is uniform.
  It is also easy to check that the maximum change to $X_i$ resulting
  from a change in a single $Y^{(i)}_{j}$ is that one element is replaced by
  another.
  
  Therefore, we can apply Theorem~\ref{concentration} if we consider
  $f(\Xvec)$ as a function of all the independent variables $\{Y^{(i)}_j\}$.
  If $a_i<\card{A_i}/2$, we can represent $X_i$ by its
  complement; this justifies the term $\min\{a_i,\card{A_i}-a_i\}$ in
  the theorem statement.
\end{proof}

We next apply these concentration inequalities to show that certain events
are very likely in our probability spaces.

\begin{thm}\label{regularity}
  The following are true for sufficiently small $\eps>0$.
  \begin{itemize}
  \item[(a)] Suppose that $(m,n,p)$ and $(m,n,k/mn)$ are $(a,\eps)$-acceptable.
  Then
  \[
     \Prob_\D\( (\Svec,\Tvec)\text{ is $\eps$-regular} \)
     = 1 - \oo(1)
  \]
  for $\D$ being any of $\G_p$, $\G_k$, $\I_p$, $\B_p$, $\B_k$, or $\V_p$.
  The same is true for $m=n$ when $\D$ is any of
  $\digG_p$, $\digG_k$, $\digI_p$, $\digB_p$, $\digB_k$, or $\digV_p$.
  \item[(b)] If\/ $\tvec\in I_m^n$ is $\eps$-regular, and
  $(m,n,\lambda_m(\tvec))$ is $(a,\eps)$-acceptable, then
  \[
     \Prob_\D\( \Svec\text{ is $\eps$-regular} \)
     = 1 - \oo(1).
  \]
  for $\D$ being $\G_\tvec$ or $\B_\tvec$.
  The same is true for $m=n$ when $\D$ is either of
  $\digG_\tvec$ or $\digB_\tvec$.
 \end{itemize}
\end{thm}
\begin{proof}
  By symmetry, we need only show that $\Svec$ is almost always
  $\eps$-regular.

  In the case that $\D$ is $\G_p$ or $\I_p$, each $S_i$ has the binomial
  distribution $\Bin(n,p)$, and $K$ has the distribution
  $\Bin(mn,p)$. Therefore, by Corollary~\ref{concentrationcor},
  \begin{align}
     \Prob_\D\( \abs{S_i-p n} \ge n^{1/2+\eps/2} \) 
        &= \oo(1),\quad i=1,\ldots,m, \notag \\
     \Prob_\D\( \abs{\varLambda -p}
             \ge n^{-1+2\eps} \) &= \oo(1), \label{Scon}
  \end{align}
  from which it follows that
  \[
     \Prob_\D\( \Svec\text{ is $\eps$-regular} \) = 1 - \oo(1).
  \]
  The cases that $\D$ is $\G_k$, $\B_p$, or $\B_k$ follow, since these
  are the same as slices of $\G_p$ or $\I_p$ of size $n^{-O(1)}$,
  using $p=k/mn$.
  Also, the distribution of $\Svec$ in $\B_\tvec$ is the same as in $\B_k$
  for $k=\sum_{j=1}^n t_j$, so that case follows too.
  
  For $\D=\G_\tvec$, note that each $S_i$ is the sum of independent
  variables $X_1,\ldots,X_n$, where $X_j$ is a Bernoulli random variable
  with mean $t_j/m$.  The theorem thus follows using the same argument
  as we used for $\G_p$.
  
  Finally consider $\D=\V_p$.  Taking $X$ to be the indicator of the
  event that $\Svec$ is not $\eps$-regular,
  Lemmas~\ref{Vpest}--\ref{BtoV} give
  \begin{align*}
     \expect_{\V_p}(X) &= O(1)\int_0^1 K_p(p') \expect_{\B_{p'}}(X)\,dp' \\
        &= O(1) \biggl( \int_0^{p-n^{-1+\eps}}\kern-1em
               + \int_{p-n^{-1+\eps}}^{p+n^{-1+\eps}}\kern-0.3em
               + \int_{p+n^{-1+\eps}}^1\biggr) \,K_p(p') \expect_{\B_{p'}}(X)\,dp'.
  \end{align*}
  The first and third integrals are $\oo(1)$ since the tails of $K_p(p')$
  are small (recall that it is a normal density with mean $p$ and
  variance $O((mn)^{-1})$), while the second integral
  is $\oo(1)$ by the present theorem in the case~$\D=\B_{p'}$.
  (Note that if $(m,n,p)$ is $(a,\eps)$-acceptable, then all 
  $p'\in [p-n^{-1+\eps},p+n^{-1+\eps}]$ are $(a',\eps)$ for slightly
  different~$a'$.)
  
  For the digraph models, the proofs are essentially the same.
\end{proof}

\nicebreak

The following concentration results will form a key part of the proof of Theorem \ref{twosides}.
\begin{thm}\label{wriggly}
  The following are true for sufficiently small $\eps>0$.
  \begin{itemize}
  \item[(a)] Suppose that $(m,n,p)$ and $(m,n,k/mn)$ are $(a,\eps)$-acceptable.
  Then
	  \begin{align}
	     \Prob_\D\Bigl(\, \sum_{i=1}^m (S_i-n\varLambda)^2 
	        = \(1+O(n^{-1/2+2\eps})\) \varLambda(1-\varLambda)mn\Bigr)
	     &= 1 - \oo(1), \label{Sflat}\\
	     \Prob_\D\Bigl( \,\sum_{j=1}^n (T_j-m\varLambda)^2 
	        = \(1+O(m^{-1/2+2\eps})\) \varLambda(1-\varLambda)mn\Bigr)
	     &= 1 - \oo(1),\label{Tflat}
	  \end{align}
	  when $\D$ is $\G_p$ or $\G_k$.  When $m=n$, the same bounds
	  hold when $\D$ is $\digG_p$ or $\digG_k$.
  \item[(b)] If\/ $\tvec\in I_m^n$ is $\eps$-regular, and
	  $(m,n,\lambda_m(\tvec))$ is $(a,\eps)$-acceptable, then
	  \eqref{Sflat} holds when $\D$ is $\G_\tvec$, and when $m=n$
	  and $\D$ is $\digG_\tvec$.	  
  \item[(c)] If $m=n$, $(n,n,p)$ and $(n,n,k/n^2)$ are $(a,\eps)$-acceptable, then
          \begin{align} 
            \Prob_\D\Bigl( \,\sum_{i=1}^n (S_i-n\varLambda)(T_i-n\varLambda)
                 = O(n^{-1/2+2\eps})  \varLambda(1-\varLambda)n^2 \Bigr)
                 = 1 - \oo(1) \label{STflat}
          \end{align}
          when $\D$ is $\digG_p$ or $\digG_k$.
  \item[(d)] If $m=n$, $(n,n,\lambda_n(\tvec))$ is $(a,\eps)$-acceptable
          and $\tvec\in I_n^n$ is $\eps$-regular, then \eqref{STflat} holds
          when $\Tvec=\tvec$ and $\D$ is $\digG_\tvec$.
  \end{itemize}
\end{thm}
\begin{proof}
 Write $R=\sum_{i=1}^m (S_i-n\varLambda)^2$.  For $i=1,\ldots,m$
 and $j=1,\ldots,n$, let $X_{ij}$ be the indicator for an edge
 from $u_i$ to $v_j$.
 {Define $\varDelta_{ii'jj'} = (X_{ij}-X_{i'j})(X_{ij'}-X_{i'j'})$.  Then we
 have 
 \begin{align}
    \frac{1}{2m}\sum_{i,i'=1}^m\;\sum_{j,j'=1}^n \varDelta_{ii'jj'}
      &= \frac{1}{m} \sum_{i,i'=1}^m\;\sum_{j,j'=1}^n X_{ij}X_{ij'} 
                   -  \frac{1}{m} \sum_{i,i'=1}^m\;\sum_{j,j'=1}^n X_{ij}X_{i'j'} \notag \\
      &= \sum_{i=1}^m S_i^2 - \frac{1}{m}\Bigl( \sum_{i=1}^m S_i\Bigr)^2
         = R. \label{Rsum}
 \end{align}}
% which implies that
% \begin{equation}\label{Rsum}
%     R = \frac{1}{2m} \sum_{i,i'=1}^m\;\sum_{j,j'=1}^n \varDelta_{ii'jj'}.
% \end{equation}} 
 When $\D$ is either $\G_p$ or $\G_\tvec$, $X_{ij}$ is independent
 of $X_{i'j'}$ if $j\ne j'$, and $\expect_\D(X_{ij})$ is independent 
 of~$i$.  This shows that $\expect_\D(\varDelta_{ii'jj'})=0$ for $j\ne j'$,
 leaving us with
 \[
     \expect_\D(R) = \frac{1}{2m} \sum_{i,i'=1}^m\sum_{j=1}^n
         \Prob_D(X_{ij}\ne X_{i'j}).
 \]
 This gives
 \begin{align*}
     \expect_{\G_p}(R) &= pq(m-1)n, \\
     \expect_{\G_\tvec}(R) &= \frac{1}{m} \sum_{j=1}^n t_j(m-t_j).
 \end{align*}
 Now define $R^*=\sum_{i=1}^m \min\{(S_i-n\varLambda)^2,m^{1+2\eps}\}$.
 If $\Svec$ is $\eps$-regular and $S_j$ is changed by~1 for some~$j$, which changes $\varLambda$
 by $1/mn$, then
 $\min\{(S_i-n\varLambda)^2,m^{1+2\eps}\}$ changes by 
 $O(m^{1/2+\eps})$ for $i=j$ and by $O(m^{-1/2+\eps})$ for $i\ne j$.
 Consequently, $R^*$ changes by $O(m^{1/2+\eps})$.
 Applying Theorem~\ref{concentration}, we find that
  \[
   \Prob_\D\( \abs{R^*-\expect_\D(R^*)} 
           \ge \tfrac12 m^{1+\eps}n^{1/2+\eps/2} \) = \oo(1)
 \]
for $\D=\G_p$.  It also holds for $\D=\G_\tvec$, using
Theorem~\ref{setfunction} in the same way.

Now~Theorem~\ref{regularity} shows that 
 $\Prob_\D(R\ne R^*)=\oo(1)$, which implies that
 $\expect_D(R^*)=\expect_D(R)+\oo(1)$.  Therefore we can argue
 \begin{align*}
    \Prob_\D\( \abs{R-&\expect_\D(R)} \ge m^{1+\eps}n^{1/2+\eps/2} \) \\
       &\le
    \Prob_\D(R\ne R^*) 
       + \Prob_\D\( \abs{R^*-\expect_\D(R)} \ge m^{1+\eps}n^{1/2+\eps/2} \) \\
   &\le \oo(1) + \Prob_\D\( \abs{R^*-\expect_\D(R^*)}
           \ge m^{1+\eps}n^{1/2+\eps/2}+\oo(1)\)\\
   &=\oo(1).
 \end{align*}
  We also have that $\varLambda$ is fixed at the value
 $\lambda_m(\tvec)=(mn)^{-1}\sum_{j=1}^n t_j$ in $\G_\tvec$ and that
 \[
     \Prob_{\G_p} \(\abs{\varLambda-p} \ge n^{-1+2\eps}\) = \oo(1),
 \]
 by~\eqref{Scon}.
 {}From these bounds, inequality~\eqref{Sflat} follows for $\G_p$
 and~$\G_\tvec$, and~\eqref{Tflat} follows for $\G_p$ by symmetry.
 By choosing $p=k/mn$ and noting that $\G_k$ is a slice of
 size $n^{-O(1)}$ of $\G_p$, the theorem is proved for~$\G_k$ too.
 
 For $\D=\digG_p, \digG_k, \digG_\tvec$, the proofs of~\eqref{Sflat}
 and~\eqref{Tflat} follow the same pattern.
 Since~\eqref{Rsum} still holds, we can note that
 $\expect_{\digG_p}(\varDelta_{ii'jj'})=\expect_{\G_p}(\varDelta_{ii'jj'})$
 and $\expect_{\digG_\tvec}(\varDelta_{ii'jj'})
 =\expect_{\G_\tvec}(\varDelta_{ii'jj'})$
 unless $\{j,j'\}\subseteq\{i,i'\}$, to infer that
 $\expect_{\digG_p}(R)=\expect_{\G_p}(R) + O(n)$
 and $\expect_{\digG_\tvec}(R)=\expect_{\G_\tvec}(R) + O(n)$.
 This is enough to ensure that the rest of the proof continues in the same way.
 (For the record, $\expect_{\digG_p}(R) = pq(n-1)^2$.)
 
 We now prove part (d); take $\D=\digG_\tvec$, with
 $\tvec$ being $\eps$-regular and $(n,n,\lambda_n(\tvec))$
 being $(a,\eps)$-acceptable.
 We have 
  \[
   \expect_{\digG_\tvec}(S_i) = \sum_{j\ne i} \frac{t_j}{n-1}
       = \frac{\lambda n^2}{n-1} - \frac{t_i}{n-1},
 \]
 from which it follows that
 \[
     \expect_{\digG_\tvec}\Bigl(\,
      \sum_{i=1}^n (S_i-\lambda n)(t_i-\lambda n) \Bigr) 
       = - \frac{\sum_{j=1}^n (t_j-\lambda n)^2}{n-1}
       = O(n^{1+2\eps}).
 \]
In the notation of Theorem~\ref{setfunction} set $A_j=[n]\setminus \{ j \}$ and $a_j=t_j$ for each $j \in [n]$. Then in the probability space $\Xvec$, $X_j$ is the set of indices of vertices incident with $v_j$ in $\digG_\tvec$. Note $S_i=|\{ j : u_i \in X_j \}|$ and two sets being minimally different in the $j$-th component corresponds to two graphs in which one of the $t_j$ edges incident with vertex $v_j$ is incident with different vertices in $U$. This means, as $\tvec$ is $\eps$-regular, $c_j = O(n^{1/2+\eps})$ for each $j$ and we can apply Theorem~\ref{setfunction} to conclude that (d) holds.

 In the case of $\D=\digG_p$, Theorem~\ref{regularity} says that
 $\Tvec$ is $\eps$-regular with probability $1-\oo(1)$, so
 (c) is true for $\digG_p$.  Finally, $\digG_k$ is a substantial
 slice of $\digG_p$ if $p=k/n^2$, so (c) holds for 
 $\digG_k$ too.
\end{proof}

\nicebreak
\section{Proofs of the main theorems}\label{s:proofs}

In this section we will give the proofs of
the theorems and corollary stated in Section~\ref{s:main}.
The bases for our analysis are the following enumerative results
of Canfield, Greenhill and McKay~\cite{CGM,GMX}.
Also see Barvinok and Hartigan~\cite{BH} for an overlapping result.

\begin{thm}[{\cite{CGM,GMX}}]\label{enumeration}
 Let $a,b>0$ be constants such that $a+b<\frac12$.  Then there is
 a constant $\eps_0=\eps_0(a,b)>0$ such that the following is true for any 
 fixed $\eps$ with $0<\eps\le\eps_0$.
 If $(\svec,\tvec)$ is $\eps$-regular, then
 \begin{align*}
   G(\svec,\tvec) &= 
    \binom{mn}{\lambda mn}^{\!\!-1}\sprod\tprod \\
    &{}\times\exp\biggl( -\dfrac{1}{2}
       \biggl( 1 - \frac{\sum_{i=1}^m \,(s_i-\lambda n)^2}{\lambda(1-\lambda)mn}\biggr)
       \biggl( 1 - \frac{\sum_{j=1}^n \,(t_j-\lambda m)^2}{\lambda(1-\lambda)mn}\biggr)
       + O(n^{-b}) \biggr).
 \end{align*}
 Moreover, if $m=n$, then
  \begin{align*}
    \dig G(\svec,\tvec) &= 
     \binom{n^2-n}{\lambda n^2}^{\!\!-1}\dsprod\dtprod \\
     &{\kern10mm}\times\exp\biggl( -\dfrac{1}{2}
        \biggl( 1 - \frac{\sum_{i=1}^n \,(s_i-\lambda n)^2}{\lambda(1-\lambda)n^2}\biggr)
        \biggl( 1 - \frac{\sum_{j=1}^n \,(t_j-\lambda n)^2}{\lambda(1-\lambda)n^2}\biggr) \\
      &{\kern 12em} 
        - \frac{\sum_{i=1}^n \,(s_i-\lambda n)(t_i-\lambda n)}{\lambda(1-\lambda)n^2}
        + O(n^{-b}) \biggr).
  \end{align*}
\end{thm} 

We first consider $\G_p$.   Suppose that $a,b>0$ are constants
with $a+b<\frac 12$, and that
$(m,n,p)$ is $(a,\eps)$-acceptable.
According to Theorems~\ref{enumeration}, \ref{regularity} and \ref{wriggly},
and~\eqref{Scon}, there is an event $B\subseteq I_{m,n}$ such that
$\Prob_{\G_p}(B) = \oo(1)$ and, for $(\svec,\tvec)\notin B$,
\begin{align}
     \abs{K-pmn} &\le mn^{2\eps},\label{Mcons}\displaybreak[0]\\
     \Prob_{\G_p}(\stevent) 
        &= p^k q^{mn-k}  \exp\( O(n^{-b}) \)
        \binom{mn}{k}^{\!\!-1}\sprod\tprod, \notag\displaybreak[0]\\
        &= p^{2k}q^{2mn-2k}\sqrt{2\pi pqmn} \,\sprod\tprod \notag\\
        &{\kern 8em}\times\exp\biggl( \frac{(k-pmn)^2}{2pqmn} + O(n^{-b})\biggr)
               \label{Gpusual}
\end{align}
for $\sum_i s_i = \sum_j t_j = k$,
where the last step follows by Stirling's formula and, as always,
we are assuming that $\eps$ is sufficiently small.
%-----

We wish to show that~\eqref{Gpusual} closely matches the probability
in~$\V_p$.
Define $P(p,\svec,\tvec) =  \Prob_{\B_p} (\stevent)$.
By the definition of $\V_p$, we have
\[
\Prob_{\V_p} (\stevent)
  = V(p)^{-1}\!\! \int_0^1 \! K_p(p') P(p',\svec,\tvec)\,dp'.
\]
By Section~\ref{s:binomialmodels} item 2, we have
\begin{equation}\label{Pprat}
   \frac{P(p',\svec,\tvec)}{P(p,\svec,\tvec)}
   = \frac{\Prob_{\I_p}\(\sum_{i=1}^m S_i=\sum_{j=1}^n T_j\)}
              {\Prob_{\I_{p'}}\(\sum_{i=1}^m S_i=\sum_{j=1}^n T_j\)}
      \;\biggl(\frac{p'}{p}\biggr)^{\!2k}
            \biggl(\frac{1-p'}{1-p}\biggr)^{\!2mn-2k}.
\end{equation}
We will divide the integral into three parts.
Define $J_p=[p-n^{-1+3\eps},p+n^{-1+3\eps}]$.
By Lemma~\ref{Bpdenom} and~\eqref{Mcons}, for $p'\in J_p$
and $(\svec,\tvec)\notin B$, we have
\begin{equation}\label{Pratio}
\frac{P(p',\svec,\tvec)}{P(p,\svec,\tvec)}
   = \exp\biggl( \frac{2(k-pmn)}{pq} (p'-p)
                      - \frac{mn}{pq}(p'-p)^2 + O(n^{-1/2})\biggr),
\end{equation}
which gives
\[
   \int_{J_p}  K_p(p') P(p',\svec,\tvec)\,dp'
      = 2^{-1/2} P(p,\svec,\tvec)
             \exp\biggl( \frac{(k-pmn)^2}{2pqmn} + O(n^{-1/2})\biggr).
\]
To bound the integral outside $J_p$, note that 
$(p'/p)^{2k}\((1-p')/(1-p)\)^{2mn-2k}$ is increasing for
$p'\le p-n^{-1+3\eps}$ and decreasing for $p\ge p+n^{-1+3\eps}$.
Also, since the mean square of a set of numbers is at least as
large as the square of their mean, we can infer from \eqref{sumsq}
that $\Prob_{\I_{p'}}\(\sum_{i=1}^m S_i=\sum_{j=1}^n T_j\)\ge (mn+1)^{-1}$
for all~$p'$.
Since $mn\,\oo(1)=\oo(1)$, we obtain from~\eqref{Pprat} that
\[
   \int_{[0,1]\setminus J_p}  K_p(p') P(p',\svec,\tvec)\,dp'
   = \oo(1) P(p,\svec,\tvec).
\]
Recalling Lemma~\ref{Vpest}, we conclude that
\begin{align*}
   V(p)^{-1} \int_0^1  K_p(p') &P(p',\svec,\tvec)\,dp' \\
      &= 2^{-1/2} P(p,\svec,\tvec)
             \exp\biggl( \frac{(k-pmn)^2}{2pqmn} + O(n^{-1/2})\biggr),
\end{align*}
which matches~\eqref{Gpusual} when the value of $P(p,\svec,\tvec)$
given by Lemma~\ref{Bpdenom} is substituted.
This completes the proof of the first claim of
Theorem~\ref{twosides}(a).
The next two claims follow on summing the first claim
over all $(\svec,\tvec)$.
For the variance, we can apply the formula for the expectation to argue
\begin{align*}
   \Var_{\G_p}(X) &= \min_{\mu\in\Reals} \expect_{\G_p}(X-\mu)^2 \\
     &= \min_{\mu\in\Reals}\, \( \oo(1) \max_{(\svec,\tvec)}\, (X-\mu)^2
           + (1+O(n^{-b})) \expect_{\V_p}(X-\mu)^2 \)  \\
     &= \min_{\mu\in\Reals}\, \( \oo(1) \max_{(\svec,\tvec)} X^2
           + (1+O(n^{-b})) \expect_{\V_p}(X-\mu)^2 \) \\
     &= \oo(1) \max_{(\svec,\tvec)} X^2
          + (1+O(n^{-b})) \min_{\mu\in\Reals} \expect_{\V_p}(X-\mu)^2 \\
     &= \oo(1) \max_{(\svec,\tvec)} X^2
          + (1+O(n^{-b})) \Var_{\V_p}(X).
\end{align*}
For the third line we have used the obvious fact that the minimum in
the first line occurs somewhere in the interval $[\min X,\max X]$.

\medskip

The proof of Theorem~\ref{twosides}(b) is the same.
To prove Theorem~\ref{twosides}(c), note that
according to Theorems~\ref{regularity}, \ref{wriggly} and \ref{enumeration},
there is an event $B\subseteq I_{m,n}$ such that
$\Prob_{\G_k}(B) = \oo(1)$ and, for $(\svec,\tvec)\notin B$,
\[
   \Prob_{\G_k}(\stevent) 
        = \exp\( O(n^{-b}) \) \binom{mn}{k}^{\!\!-2}\sprod\tprod,
\]
which matches $\Prob_{\B_k}(\stevent)$
up to the error term.
Similarly for Theorem~\ref{twosides}(d).

Theorem~\ref{oneside} follows from a similar argument, on noting that
the $\eps$-regularity of~$\tvec$ implies 
\[
    \sum_{j=1}^n \,(T_j-\lambda m)^2 \le n^{2+2\eps}
       \le m^{4\eps}\lambda(1-\lambda)mn.
\]

\medskip

Finally, we prove Corollary~\ref{vanishing} for $\D=\G_p$, which
is representative of the four cases.
In view of Theorem~\ref{twosides}, it will suffice to prove that
\begin{equation}\label{toprove}
   \Prob_{\V_p}(E) \le \oo(1) + o(1)\ssqrt{\Prob_{\B_p}(E)}
\end{equation}
if $\Prob_{\B_p}(E)\to 0$.
Define
\[
   y=\max\Bigl\{n^\eps,\ssqrt{-\log(\Prob_{\B_p}(E))
           -\tfrac12\log\log(\Prob_{\B_p}(E))}\,\Bigr\}
\]
and
\[
  \hat E = \bigl\{ (\svec,\tvec) \in E : \abs{K-pmn}\le y\sqrt{pqmn}\, \bigr\}.
\]
By a suitable normal approximation of the binomial distribution,
such as~\cite[Thm.\,3]{littlewood},
$\Prob_{\G_p}(E\setminus\hat E)=O(e^{-y^2/2}/y)$,
so by Theorem~\ref{twosides},
$\Prob_{\V_p}(E\setminus\hat E)=\oo(1)+O(e^{-y^2/2}/y)$.
Also note that
\begin{equation}\label{tails}
    \int_{\abs{p'-p}>y\sqrt{pq/2mn}} K_p(p')\, dp' = O(e^{-y^2/2}/y).
\end{equation}
Therefore, since $V(p)^{-1} = 1+\oo(1)$ by Lemma~\ref{Vpest},
\[ 
    \Prob_{\V_p}(E) = \oo(1) + O(e^{-y^2/2}/y) 
           + \int_{\abs{p'-p}\le y\sqrt{pq/2mn}} 
           K_p(p') \Prob_{\B_{p'}}(\hat E)\, dp.
\]
Now note that, by~\eqref{Pratio}, for $\abs{p'-p}\le y\sqrt{pq/2mn}$
and $\abs{k-pmn}\le y\sqrt{pqmn}$ we have
\begin{align*}
  &\frac{\Prob_{\B_{p'}}(\stevent)}{\Prob_{\B_p}(\stevent)}\\
    &{\kern20mm} \le \exp\biggl( \frac{2(k-pmn)}{pq} (p'-p)
                      - \frac{mn}{pq}(p'-p)^2 + O(n^{-1/2})\biggr) \\
    &{\kern20mm} \le (1+O(n^{-1/2})) e^{y^2/2}
\end{align*}
and so 
\[
  \Prob_{\B_{p'}}(\hat E) \le (1+O(n^{-1/2})) e^{y^2/2} \Prob_{\B_p}(\hat E).
\] 
Since $\int K_p(p')\,dp < 1$, we have proved that
\[
    \Prob_{\V_p}(E) \le \oo(1) + O(e^{-y^2/2}/y)
         + (1+o(1)) e^{y^2/2} \Prob_{\B_p}(E),
\]
which gives \eqref{toprove} when the value of $y$ is substituted. To prove the statement for the case $\D=\B_p$, redefine $y$ and $\hat E$ by replacing each instance of $\B_p$ with $\V_p$. and then proceed in the same fashion (although in this case because of the direction of the inequality it is enough to note that the tails of the integral in \eqref{tails} are positive;
we do not need to show an upper bound as in the above proof for $\D=\G_p$). 

%============================================

\section{Concluding remarks}

A theorem similar to Theorem~\ref{enumeration} holds also in the
sparse domain.  This was shown by Greenhill, McKay and Wang
in the case that $(\max_i s_i)(\max_j t_j) = o\((\sum_i s_i)^{2/3}\)$~\cite{GMW}.
That theorem can be used to develop a parallel theory of
degree sequences in that domain, though some of the methods
used in this paper must be replaced.
However the lack of a precise enumeration in the gap between
the sparse domain and the dense domain of Theorem~\ref{enumeration}
currently thwarts a theory which spans both the sparse and dense
domains.

\nicebreak

\end{document}